\documentclass[12pt]{amsart}
\usepackage{amscd} 
\usepackage{amsfonts} 
\usepackage{amssymb} 
\usepackage{latexsym} 
\usepackage[arrow, matrix, curve]{xy}

\newcommand{\ncm}{\newcommand}

\newtheorem{theorem}{Theorem}[section]
\newtheorem{prop}[theorem]{Proposition}
\newtheorem{lemma}[theorem]{Lemma}
\newtheorem{cor}[theorem]{Corollary}
\newtheorem{lem&def}[theorem]{Lemma \& Definition}
\newtheorem{definition}[theorem]{Definition}
\newtheorem{example}[theorem]{Example}
\newtheorem{remark}[theorem]{Remark}

\def\M{\mathcal{M}}

\def\C{\mathbb{C}\,} 
\def\Z{\mathbb{Z}\,} 
\def\R{\mathbb{R}\,} 
\def\N{\mathbb{N}\,}

\ncm{\End}{\mbox{\rm End}\,}
\def\Ann{\mbox{\rm Ann}}

\def\|{\, | \,}

\def\id{\mbox{\rm id}}

\def\into{\hookrightarrow}

\ncm{\rarr}[1]{\stackrel{#1}{\longrightarrow}}
\ncm{\larr}[1]{\stackrel{#1}{\longleftarrow}}
\def\cop{\Delta}

\def\eps{\varepsilon}

\def\-2{_{(-2)}}
\def\-1{_{(-1)}}
\def\0{_{(0)}}
\def\1{_{(1)}}
\def\2{_{(2)}}
\def\3{_{(3)}}
\def\4{_{(4)}}
\def\i+1{_{(i+1)}}
\def\n{_{(n)}}
\def\n+1{_{(n+1)}}

\def\du1{\hat 1}

\begin{document}
\title[Subgroup depth and twisted coefficients]{Subgroup depth and twisted coefficients}
\author[A.~Hernandez, L.~Kadison and M.~Szamotulski]{Alberto Hernandez, Lars Kadison and Marcin Szamotulski} 
\address{Departamento de Matematica \\ Faculdade de Ci\^encias da Universidade do Porto \\ 
Rua Campo Alegre 687 \\ 4169-007 Porto} 
\email{ahernandeza079@gmail.com, lkadison@fc.up.pt, mszamot@gmail.com} 
\thanks{}
\date{} 

\begin{abstract}
Danz computes the depth of certain twisted group algebra extensions in \cite{D}, which are less than the values of the depths of the corresponding untwisted group algebra extensions in \cite{BKK}.  
In this paper, we show that the closely related h-depth of any group crossed product algebra extension is less than or equal to 
the h-depth of the corresponding (finite rank) group algebra extension.  
A convenient theoretical underpinning to do so is provided by the 
entwining structure of a right $H$-comodule algebra $A$ and a right $H$-module coalgebra $C$ for a Hopf algebra $H$. 
Then $A \otimes C$ is an $A$-coring, where corings have a notion of depth extending h-depth. This
coring is Galois in certain cases where 
$C$ is the quotient module $Q$ of a coideal subalgebra $R \subseteq H$.  We note that this applies 
for the group crossed product algebra extension, so that the depth of this Galois coring is less than the h-depth of $H$ in $G$.  Along the way, we show that subgroup depth behaves exactly like combinatorial depth with respect to the core of a subgroup, and extend results in \cite{K2013} to coideal subalgebras of finite dimension.    
\end{abstract} 
\maketitle

\section{Introduction and Preliminaries}
\label{one}
Subgroup depth $d_0(H,G)$ of a subgroup $H$ in a finite group $G$ is introduced in \cite{BKK} as the minimum depth of the induction-restriction table of irreducible characters of $H$ and $G$, a matrix
of nonnegative integers with nonzero rows and columns.  As the matrix of induction $K_0(\C H) \rightarrow K_0(\C G)$, the notion of depth also occurs in fields of topological algebra in various guises.  In ring theory, the minimum depth $d(B,A)$ of a subring $B \subseteq A$ is introduced in \cite{BDK} in terms of the natural bimodule ${}_BA_B$ and its tensor powers (and for even depth, tensored one more time by
the bimodule ${}_BA_A$ or ${}_AA_B$).  The depth $d_0(H,G)$ is recovered in \cite{BDK} by letting
$B$, $A$ be group algebras over a field of characteristic zero; indeed, \cite{BDK} shows that depth $d(kH,kG)$  depends only on the characteristic of the field $k$ \cite{BDK}.  

The more general definition also allows one to consider depth of integral group ring extensions $\Z H \subseteq \Z G$ and their minimum depth $d_{\Z}(H,G)$.
In addition, a combinatorial depth $d_c(H,G)$ is introduced in \cite{BDK} by using G-set 
analogues of balanced tensors and bimodules.  The following string of inequalities is from \cite[4.5]{BDK}:  $$d_0(H,G) \leq d_p(H,G) \leq d_{\Z}(H,G) \leq d_c(H,G) \leq 2|G: N_G(H) |.$$ 

In \cite{LK2011}, h-depth of a subring pair $B \subseteq A$ is introduced by the same process as in the definition of depth but focussing on the natural $A$-$A$-bimodules of the tensor powers of $A$ relative to $B$.   The minimum h-depth $d_h(B,A)$
is closely related to $d(B,A)$ by the inequality $| d(B,A) - d_h(B,A) | \leq 2$.  Its definition  in~\ref{def-depth} suggests the notion of h-depth is  natural and almost unavoidable when considering subgroup depth.  In this paper,  h-depth provides a natural transition (in Section~2) to depth of an $A$-coring, which is a notion of coalgebra generalized to $A$-bimodules
where the comultiplication and counits are $A$-bimodule morphisms.

In \cite{D} the subring depth of twisted group algebras of the permutation groups $\Sigma_n$
are computed as an intriguing contrast to  the untwisted case in \cite{BKK}, where it was shown that $d_0(\Sigma_n, \Sigma_{n+1}) = 2n-1$.
In \cite{D} it was shown that with $\alpha$ the nontrivial $2$-cocycle  (representing the nonzero element in  $H^2(\Sigma_n, \C^{\times})$), the twisted complex group algebra extensions have minimum depth  
\begin{equation}
\label{eq: Danz}
d(\C_{\alpha}\Sigma_n, \C_{\alpha}\Sigma_{n+1}) = 2(n - \lceil \frac{\sqrt{8n+1} - 1}{2} \rceil) + 1.
\end{equation}

Note that $d(\C_{\alpha}\Sigma_n,\C_{\alpha}\Sigma_{n+1}) \leq d_0(\Sigma_n, \Sigma_{n+1})$
with a difference that goes to infinity as $n \rightarrow \infty$. The same is true of the alternating group series $A_n$, where  $d_0(A_n,A_{n+1}) = 2(n - \lceil \sqrt{n}\, \rceil) +1$ $\geq d(C_{\alpha}A_n, \C_{\alpha}A_{n+1})$, the last depth also equal to the right-hand side of Eq.~(\ref{eq: Danz}); cf.\ \cite[appendix]{BKK}. 
Given a subgroup $H \leq G$, we show in this paper that  the crossed product of a twisted $G$-algebra $A$  with subalgebra
$A \#_{\sigma} H$ has  h-depth less than or equal to the h-depth of the corresponding group algebra extension:
i.e., we establish 
$$d_h(A \#_{\sigma} H,A \#_{\sigma} G) \leq d_h(kH, kG)$$
in Eq.~(\ref{eq: hdepthineq}) in Section~4 below.  
We will also extend an equality for h-depth of a Hopf subalgebra in terms of depth of its quotient module \cite{K2013} to the  equalition 
for h-depth of a left coideal subalgebra of a Hopf algebra $H$ in  Corollary~\ref{cor-cue}.
Our method is to define depth of corings so that the depth of the Sweedler coring of a ring extension is its h-depth, and apply (Doi-Koppinen) entwining structures that are Galois corings.  

 In Corollary~\ref{cor-core} below we show that $d_h(H,G) = d_h(G/N,H/N)$ where $N$ is the core of a subgroup $H$
in a finite group $G$ (over any ground field).  In Corollary~\ref{cor-cores} we note that subgroup depth  behaves precisely like combinatorial depth with repect to this subgroup $N$, after proving 
that subring depth is preserved by quotienting of relatively nice ideals (Theorem~\ref{prop-sigma}), or
better yet, by 
relatively nice Hopf ideals (Proposition~\ref{prop-crop}).  A final application is to a left coideal subalgebra $R$ of a finite-dimensional Hopf algebra $H$, which is left normal iff a nonzero right integral in $R$ is a normal element in $H$ (Theorem~\ref{th-leftrightleft}). 

\subsection{Preliminaries on subalgebra depth} Let $A$ be a unital associative algebra over a field $k$. The category of  modules over $A$ will be denoted
by $\M_A$.  (For finite-dimensional $A$, the notation $\M_A$ denotes the category of finite-dimensional modules.)
Two modules $M_A$ and $N_A$ are \textit{similar} (or H-equivalent) if $M \oplus * \cong q \cdot N := N \oplus \cdots \oplus N$ ($q$ times) and $N \oplus * \cong r \cdot M$
for some $r,q \in \N$.  This is briefly denoted by $M \| q \cdot N$ and $N \| r \cdot M$ for some $q,r \in \N$  $\Leftrightarrow$ $M \sim N$. Recall that similar modules have Morita equivalent endomorphism rings.  

 Let $B$ be a subalgebra of $A$ (always supposing $1_B = 1_A$).  Consider the natural bimodules ${}_AA_A$, ${}_BA_A$, ${}_AA_B$ and
${}_BA_B$ where the last is a restriction of the preceding, and so forth.  Denote the tensor powers
of ${}_BA_B$ by $A^{\otimes_B n} = A \otimes_B \cdots \otimes_B A$ ($n$ times $A$) for $n = 1,2,\ldots$, which is also a natural bimodule over  $B$ and $A$ in any one of four ways;     set $A^{\otimes_B 0} = B$ which is only a natural $B$-$B$-bimodule.

\begin{definition}
\label{def-depth}
Suppose that $A^{\otimes_B (n+1)}$ is similar to $A^{\otimes_B n}$ as natural $X$-$Y$-bimodules
for subrings $X$ and $Y$ of $A$ and $n \in \N$.  One says that subalgebra $B \subseteq A$ has  
\begin{itemize}
\item \textbf{depth} $2n+1$ if $X = B = Y$ for $n \geq 0$;
\item \textbf{left depth} $2n$ if $X = B$ and $Y = A$,
\item \textbf{right depth} $2n$ if $X = A$ and $Y = B$,
\item \textbf{h-depth} $2n-1$ if $X = A = Y$,
\end{itemize}
 for $n \geq 1$. 

Note that if $B \subseteq A$ has h-depth $2n-1$, the subalgebra has (left or right) depth $2n$ by restriction of modules.  Similarly, if $B \subseteq A$ has depth $2n$, it has depth $2n+1$.  If $B \subseteq A$ has depth $2n+1$, it has depth $2n+2$ by tensoring either $-\otimes_B A$ or $A \otimes_B -$ to $A^{\otimes_B (n+1)} \sim A^{\otimes_B n}$.     Similarly, if $B \subseteq A$ has left or right depth $2n$, it has h-depth $2n+1$.  Denote the \textbf{minimum depth}
of $B \subseteq A$  by $d(B,A)$ \cite{BDK}.  Denote the \textbf{minimum h-depth} of $B \subseteq A$ by $d_h(B,A)$ \cite{LK2011}.  Note that
$d(B,A) < \infty$ $\Leftrightarrow$ $d_h(B,A) < \infty$; if so, \cite{LK2011} shows that
\begin{equation}
-2 \leq  d(B,A) - d_h(B,A) \leq 1.
\end{equation}
\end{definition}

 For example, $B \subseteq A$ has depth $1$ iff ${}_BA_B$ and ${}_BB_B$ are similar \cite{BK2, LK2012}.  In this case, one deduces the following algebra isomorphism, 
\begin{equation}
\label{eq: d1eq}
A \cong B \otimes_{Z(B)} A^B,
\end{equation}
 where
$Z(B), A^B$ denote the center of $B$ and centralizer of $B$ in $A$.

Another example is that $B \subset A$ has right depth $2$ iff ${}_AA_B$
and ${}_A A \otimes_B A_B$ are similar.  If $A = \C G$ is a group algebra of a  finite group $G$ and $B = \C H$ is a group algebra of a subgroup $H$ of $G$, then $B \subseteq A$ has right depth $2$ iff $H$ is a normal subgroup of $G$ iff $B \subseteq A$ has left depth $2$; a similar statement is true for a Hopf subalgebra $R \subseteq H$ of finite index and over any field \cite{BK}. 

 Now let $H$ be a Hopf algebra with counit $\eps: H \rightarrow k$, antipode $S: H \rightarrow H$ and
coproduct $\cop: H \rightarrow H \otimes H, \cop(h) = h\1 \otimes h\2$ (Sweedler notation suppressing summations).  Let $A$ a \textit{right $H$-comodule algebra}, i.e., an $H$-comodule
with coaction $\rho: A \rightarrow A \otimes H, \rho(a) = a\0 \otimes a\1$ that is a unital algebra homomorphism.  The
coinvariants $\{ b \in A \, | \, \rho(b)  = b \otimes 1_H \}$ form a subalgebra
$B$; one says $A$ is an \textit{$H$-Galois extension} of $B$ if the Galois mapping $A \otimes_B A \rightarrow A \otimes H$ given by $x \otimes y \mapsto  x y\0 \otimes y\1$ is bijective.  Note that
the Galois mapping is an isomorphism of natural $A$-$B$-bimodules:  if $\dim H = n$,
then $A \otimes_B A \cong n \cdot {}_AA_B$, and $A \supseteq B$ has depth $2$.  
There is the following remarkable converse growing out of Ocneanu's  ideas in subfactor theory with detailed papers by Szymanski, Longo, Nikshych-Vainerman and others:

  \begin{theorem}
Let $A \supseteq B$ be a Frobenius algebra extension with $1$-dimensional centralizer and surjective Frobenius homomorphism.  If $d(B,A) \leq 2$,
then $A$ is a Hopf-Galois extension of $B$.  
\end{theorem}
 
The notion of Frobenius extension is defined in Example~\ref{example-Frobenius}; a short proof-with-references in \cite{LK2004}.  The surjectivity condition ensures
that $A_B$ is a generator (and conversely \cite{LK2012}).    This theorem requires particularly the use of the depth two condition for the construction of a Hopf algebra
structure on $\End {}_BA_B$ or its dual Hopf algebra structure on $\End {}_AA \otimes_B A_A \cong (A \otimes_B A)^B$ \cite{KS}.  (However,  if $d(B,A) = 1$, none of these difficulties arise, as Eq.~(\ref{eq: d1eq}) forces $A = B$ with one-dimensional center, the trivial Hopf algebra.)  The stringent condition on the centralizer $A^B$ may be relaxed if
one considers more general Hopf algebras and their Galois coactions (such as weak Hopf algebras and
Hopf algebroids) \cite{KS}.

Note that one always has 
\begin{equation}
\label{eq: divides}
A^{\otimes_B n} \| A^{\otimes_B (n+1)}
\end{equation} as natural $A$-bimodules for all $n \geq 2$: one applies the unit mapping and multiplication to obtain a split monic (or split epi). For  $n =1$, though it holds for the other three of the bimodule structures, it
is not generally true  as $A$-$A$-bimodules,  $A \| A \otimes_B A$ being the separable extension condition on $B \subseteq A$.  
Now $A \otimes_B A \| q \cdot A$ as $A$-$A$-bimodules for some $q \in \N$
is the H-separability condition and implies $A$ is a separable extension of $B$ by Hirata, cf.\ \cite[2.6]{NEFE}.  Somewhat similarly, ${}_BA_B \| q  \cdot {}_BB_B$ implies
${}_BB_B \| {}_BA_B$ \cite{LK2012}. It follows that subalgebra depth and h-depth may be equivalently defined by replacing the similarity bimodule conditions for depth and h-depth in Definition~\ref{def-depth} with the corresponding bimodules on 
\begin{equation}
\label{eq: def}
A^{\otimes_B (n+1)} \| q \cdot A^{\otimes_B n}
\end{equation}
 for some positive integer $q$ \cite{BDK, LK2011, LK2012}.  

For example, for the permutation groups $\Sigma_n < \Sigma_{n+1}$
and their corresponding group algebras over any
commutative ring $K$, one has depth $d_K(\Sigma_n,\Sigma_{n+1}) = 2n-1$ \cite{BDK}. Depths of subgroups in $PGL(2,q)$, Suzuki groups, twisted group algebra extensions and Young subgroups of $\Sigma_n$ are computed in \cite{F, HHP, D, FKR}.    If $B$ and $A$ are semisimple complex algebras, the minimum odd depth is computed from powers of an order $r$ symmetric matrix with nonnegative entries $\mathcal{S} := MM^T$ where $M$ is the inclusion matrix
$K_0(B) \rightarrow K_0(A)$ and $r$ is the number of irreducible representations of $B$ in a basic set of $K_0(B)$; the depth is $2n+1$ if $\mathcal{S}^n$ and $\mathcal{S}^{n+1}$ have an equal number of zero entries \cite{BKK}. It follows that the subalgebra pair of semisimple complex algebras $B \subseteq A$   always has finite depth.

Similarly, the minimum h-depth of $B\subseteq A$ is computed from
powers of an order $s$ symmetric matrix $\mathcal{T} = M^TM$, where $s$ is the rank of $K_0(A)$; the h-depth is $2n+1$ if $\mathcal{T}^n$ and $\mathcal{T}^{n+1}$ have an equal number of zero entries (equivalently, letting $\mathcal{T}^0 = I_{s \times s}$, one has 
\begin{equation}
\label{matrixinequality}
\mathcal{T}^{n+1} \leq q\mathcal{T}^n
\end{equation}
 for some $q \in \N$) \cite[Section 3]{LK2012}.

\subsection{Depth of Hopf subalgebras, coideal subalgebras and modules in a tensor category $\mathcal{M}_H$}  Let $ H$ be a l Hopf algebra over an arbitrary field $k$. Let $R \subseteq H$ be a Hopf subalgebra, so that $\cop(R) \subseteq R \otimes R$ and the antipode satisfies $S(R) = R$.   It was shown in \cite[Prop.\ 3.6]{K2013} that the tensor powers  of $H$ over $R$, denoted by $H^{\otimes_R n}$,  reduce to tensor powers of the generalized quotient $Q := H/ R^+H$ as follows: 
\begin{equation}
\label{eq: diagaction}
H^{\otimes_R n} \stackrel{\cong}{\longrightarrow} H_{\cdot} \otimes Q_{\cdot}^{\otimes (n-1)}
\end{equation} which
for $n = 2$ is given by $x \otimes_R y \mapsto xy\1 \otimes \overline{y\2}$; see \cite[Eq. (21)]{HKY} for the
straightforward extension of this to all $n$. 
 The isomorphism  is an $H$-$H$-bimodule isomorphism where the left $H$-module structures are the natural endpoint actions (as well as the right action to the left), and the right $H$-module structure on $H \otimes Q \otimes \cdots \otimes Q$ is given by the diagonal action of $H$: 
$$h'(y \otimes q_1 \otimes \cdots \otimes q_{n-1}) \cdot h = h'yh\1 \otimes q_1 h\2 \otimes \cdots \otimes q_{n-1} h_{(n)}.$$

The following proposition directly makes use of the isomorphism for each $n \geq 2$.  Given a subalgebra pair $U \supseteq T$, observe that the bimodule
${}_UU_T$ is projective iff the multiplication mapping $U \otimes T \rightarrow T$, $u \otimes t \mapsto ut$, is a split
epi of $U$-$T$-bimodules iff there is an element $e = e^1 \otimes e^2 \in U \otimes T$ such
that $e^1 e^2 = 1_U$ and $te = et$ for each $t \in T$ (a so-called ``right relative separable tower'' of algebras
$U \supseteq T \supseteq k1_U$).  

\begin{prop}
Suppose $H$ is a finite-dimensional Hopf algebra with $R$ and $Q$ as above, with intermediate
Hopf subalgebras $H \supseteq U \supseteq T \supseteq R$. 
If ${}_UU_T$ is a projective bimodule, then $d(R,H) < \infty$. In particular, the depth is finite if either $H$ or $R$ is semisimple \cite{LK2012, K2013}.   
\end{prop}
\begin{proof}
In a finite tensor category, such as $\mathcal{M}_H$ with  the diagonal action, $P \otimes_k X$ is projective if  $P$ is projective and $X$ is any $H$-module \cite[Prop.\ 2.1]{EO} (its proof does not need $k$ to be algebraically closed). 
The proof below will also require the notion of  tensor product Hopf algebra $H \otimes K$ of two Hopf algebras $H$, $K$, as well as the Hopf opposite algebra (with antipode $S^{-1}$) \cite{Ma}.    Then Eq.~(\ref{eq: diagaction}) shows (by restriction) that
each $H^{\otimes_R n}$ is a projective $U$-$T$-bimodule, equivalently projective $H \otimes T^{\rm op}$-module, for one  extends $Q$ to an  (left-sided trivial) $U$-$T$-bimodule via $uqt = \eps(u)qt$.  Since $R^e = R \otimes R^{\rm op}$ is a Hopf subalgebra in the finite-dimensional $U \otimes T^{\rm op}$, and therefore a free extension, it follows that each $H^{\otimes_R n}$ is projective
as a natural $R$-$R$-bimodule. 
 The Krull-Schmidt Theorem, Eq.~(\ref{eq: divides}) and the fact that there are finitely many projective indecomposable isoclasses entail that $H^{\otimes_R (N+1)} \sim H^{\otimes_R (N+2)}$ for $N$ equal to this number (or the number of nonisomorphic simple $H$-$R$-bimodules).  Thus, $d(R,H) \leq 2N + 3$. 

A semisimple Hopf algebra $H$ (or $R$) is a  separable algebra, separability   being characterized by the condition ${}_{H^e}H$ (or ${}_RR_R$)  is 
projective.  The finite depth follows by letting $U = T = H$ (or $U = T = R$).    
\end{proof}

 The bimodule isomorphism in Eq.~(\ref{eq: diagaction}) shows quite clearly
that the following definition applied to $Q$ will be of interest to computing $d(R,H)$.  Let $W$ be a right $H$-module and
$T_n(W) :=
 W \oplus W^{\otimes 2} \oplus \cdots \oplus W^{\otimes n}$.  

\begin{definition}
A module $W$ over a Hopf algebra $H$ has \textbf{depth} $n$ if $T_{n+1}(W) \| t \cdot T_n(W) $
for some positive integer $t$, 
and depth $0$ if $W$ is isomorphic to a direct sum of copies of $k_{\eps}$, where $\eps$ is the counit.
Note that this entails that $W$ also has depth $n+1$, $n+2$, $\ldots$.  Let $d(W, \M_H)$ denote
its \textbf{minimum depth}.  If $W$ has a finite depth, it is said to be an \textit{algebraic module}.  
If $W$ is an $H$-module coalgebra, or $H$-module algebra, the condition of depth $n$ simplifies
to $W^{\otimes (n+1)} \| t \cdot W^{\otimes n}$ for some $t \in \Z_+$ and all $n \in \N$,  where
$W^{\otimes 0}$ denotes $k_{\eps}$.  
  \end{definition}

\begin{lemma}
\label{lemma-modI}
Let  $I$ be a Hopf ideal in a Hopf algebra $H$.  Suppose $I$ is contained in the annihilator ideal of an $H$-module $W$.
Then depth of $W$ is the same over $H$ or $H/I$.
\end{lemma}
\begin{proof}
The lemma is proven by noting that a Hopf ideal $I$ in $\Ann_H W$ is contained in the annihilator ideal of each tensor power of $W$,  since $I$ is a coideal.  Additionally, split epis as in $T_{n+1}(W) \| t \cdot T_n(W) $ descend and lift along $H \rightarrow H/I$.  
\end{proof}

Recall that the Green ring, or representation ring, of a finite-dimensional Hopf algebra $H$ over a field $k$, denoted by $A(H)$, is the free abelian group with basis consisting of  indecomposable (finite-dimensional) $H$-module isoclasses, 
with addition given by direct sum, and the multiplication in its ring structure given by the tensor product.  For example,
$K_0(H)$ is a finite rank ideal in $A(H)$.  
As shown in \cite{Fe}, a finite depth $H$-module $W$ satisfies a polynomial with integer coefficients in $A(H)$, and conversely.  Thus, an algebraic module has isoclass an algebraic element in the Green ring,
which explains the terminology.   

The main theorem in \cite[5.1]{K2013} proves that Hopf subalgebra (minimum) depth and depth of its generalized quotient $Q$ are closely related by
\begin{equation}
\label{eq: inequalityfordepth}
2d(Q,\M_R) + 1 \leq d(R,H) \leq 2d(Q, \M_R) + 2.
\end{equation}
Here one restricts $Q$ to an $R$-module, in order to obtain the better result on depth. If $R$ is a left coideal subalgebra of $H$, it is not itself a Hopf algebra, and this as a result is unavailable: in this case, we prove below (Corollary~\ref{cor-cue}) that the minimum h-depth satisfies 
\begin{equation}
\label{eq: hdepthandmoduledepth}
  d_h(R,H) = 2d(Q, \M_H) + 1. 
\end{equation}

\begin{example}
\begin{rm}
For the reader who knows something about the Drinfeld double Hopf algebra $D(H)$ of a Hopf algebra $H$ (or group $G$),
we work an example using the notation of $h \in H$ and its dual Hopf subalgebra $f \in H^*$ as subalgebras of $fh \in D(H)$ subject to relations $hf = f(S^{-1}(h\3) ? h\1) h\2$ \cite{M,Ma}.  We compute the quotient $Q$ of (coopposite) $H^*$ as a Hopf subalgebra of $D(H)$ simply by $$Q = H^* H/{H^*}^+H \cong H$$
via $\overline{fh} \mapsto f(1_H)h$ with right $H^*$-action $\overline{h}f =  f(S^{-1}(h\3)h\1) h\2$.  This gives $\overline{h} f =  \overline{h} f(1_H)$ if $H$ is cocommutative.  

Thus if $H$ is cocommutative, $d(Q, \M_{H^*}) = 0$, and by the inequality in (\ref{eq: inequalityfordepth}), 
$$ d(H^*, D(H)) \leq 2,$$
i.e., $H^*$ is a normal Hopf subalgebra in $D(H)$ (recovering some folklore with a hands-off approach). Also by Eq.~(\ref{eq: hdepthandmoduledepth}), $d_h(H^*,D(H)) = 3$ 
since $Q_{D(H)} \cong  H_H$ the regular representation, which has depth $1$ if $\dim H > 1$ \cite{K2013}.  Minimum depth $d(G, D(G))$ is analyzed in \cite{HKY}.  
\end{rm}
\end{example}

Finally we remark that if $H$ is semisimple, Eq.~(\ref{eq: hdepthandmoduledepth}) in principle computes the depth of the quotient module $Q$
in the finite tensor category $\M_H$ in terms of the symmetric matrix $\mathcal{T}$ in inequality~(\ref{matrixinequality}) (where $A = H$ and $B = R$ is semisimple
\cite[ch.\ 3]{M}): let
$d'(\mathcal{T})$ denote the least $n$ for which inequality~(\ref{matrixinequality}) holds,  then 
\begin{equation}
d(Q, \M_H) = d'(\mathcal{T}) 
\end{equation}
Thus if 
$M$ denotes the inclusion matrix $K_0(R) \rightarrow K_0(H)$, then  $2d'(\mathcal{T})+1 = d_{odd}(M^T)$ (cf.\ \cite[3.2]{LK2011}) in terms of the minimum (odd) depth of an irredundant nonnegative matrix $M$ (and its transpose $M^T$) defined in \cite[2.1]{BKK}.
\subsection{Depth of a group subalgebra pair compared with its corefree quotient pair}
Given a ring $A$ with subring $B$, say that an $A$-ideal is
\textit{relatively nice} if its intersection, the $B$-ideal $J = I \cap B$, satisfies $AJ = I$ or $JA = I$. Note that if both set equalities hold, such as when I is a relatively nice Hopf ideal in a finite-dimensional Hopf algebra with respect to a Hopf subalgebra, we are studying a type of normality condition related to the normality notion in  \cite[Section 4]{BKK}. 
Noting the canonical inclusion $B/J \into A/I$, we extend   Theorem 3.6 in \cite{KY}, where  $I$ is fully contained in $B$, to relatively nice ideals.

\begin{theorem}
\label{prop-sigma}
Let $I$ be a relatively nice ideal in $A$, where $J$ is its intersection with $B$. Minimum depth satisfies $d(B/J, A/I) \leq d(B, A)$.  
\end{theorem}
\begin{proof}
If $d(B,A) = \infty$, there is nothing to prove.  
Assume that there is a split monomorphism $\sigma: A^{\otimes_B (n+1)} \into q \cdot A^{\otimes_B n}$ of (natural $B$- or $A$-) bimodules for some $q, n \in \N$,
where we use balanced multilinear notation for the arguments $a_i \in A$ ($i = 1, \ldots, n+1$).  Let $\pi$ denote the canonical surjection $A \rightarrow A/I$,
where $\pi(a_i) = \overline{a_i}$.
Define $$\overline{\sigma}(\overline{a_1},\ldots,\overline{a_{n+1}}) = q \cdot \pi^{\otimes n}\sigma(a_1,\ldots,a_{n+1}) = (\pi^{\otimes n}\sigma_i(a_1,\cdots,a_{n+1}))_{i=1}^q. $$
We must show that for any $i = 1,\ldots,q$ and $y \in I$, $$\pi^{\otimes n}\sigma_i(a_1, \cdots,y,\cdots,a_{n+1}) = 0$$ to see that $\overline{\sigma}$ is well-defined.  
Suppose without loss of generality that $I = JA$, so there are finitely many $x_{j_{r+1}} \in J$ and $a_{j_{r+1}}\in A$ such that $y = \sum_{j_{r+1}} x_{j_{r+1}}a_{j_{r+1}}$, where we note  $$\sigma_i(a_1, \cdots,a_r, y,\ldots,a_{n+1}) = 
\sum_{j_{r+1}}\sigma(a_1,\ldots, a_rx_{j_{r+1}},a_{j_{r+1}},a_{r+1}, \ldots,a_{n+1}),$$
but $a_r x_{j_{r+1}} \in I = JA$. Then this process may be repeated $r$ times, until we have the terms
$$\sigma_i(a_1,\ldots,y,\ldots,a_{n+1}) = $$
$$\ \ \ \ \ \sum_{j_1,\cdots,j_{r+1}}^N x_{j_1 \cdots j_{r+1}}\sigma(a_{j_1 \cdots j_{r+1}},\ldots,a_{j_r j_{r+1}}, a_{j_{r+1}},a_{r+1},\ldots,a_{n+1})$$
where $x_{j_1 \cdots j_{r+1}} \in J$, all terms mapping to zero under $\pi^{\otimes n}$.  Similarly a splitting for $\sigma$ descends to $q \cdot (A/I)^{\otimes_{B/J} n}
\rightarrow (A/I)^{\otimes_{B/J} (n+1)}$, a splitting for $\overline{\sigma}$.    The inequality of minimum depths follows as a consequence of the characterization of depth in (\ref{eq: def}).
\end{proof}

\begin{remark}
\begin{rm}
The proof of Theorem~\ref{prop-sigma} may be adapted slightly to show that relative Hochschild cochain groups of a subring pair
$B \subseteq A$ with coefficients in an $A$-bimodule $M$ \cite{H} with left and right annihilator containing a relatively nice ideal $I$
in $A$ satisfies $C^n(A,B;M) \cong C^n(A/I,B/J; M)$ induced by the canonical algebra epi $A \rightarrow A/I$, 
a cochain isomorphism for all $n \geq 0$.
\end{rm}
\end{remark}

For the following corollary, assume  $I$ is a relatively nice Hopf ideal in
a Hopf algebra  $H$.  
\begin{prop}
\label{prop-crop}
The minimum depth satisfies $$d(R/J, H/I) \leq d(R,H) \leq d(R/J, H/I) + 1.$$ Moreover, if $d(R/J,H/I)$
is even, the two minimum depths are equal.  
\end{prop}
\begin{proof}
Note that $J = R \cap I$ is a Hopf ideal in $R$ and satisfies both $HJ = I = JH$ by an application of the antipode. 
 Let $Q = H/R^+H$ be the quotient module of $R \subseteq H$, and
$Q'$ be the corresponding quotient module of $R/J \subseteq H/I$.  Since $J \subseteq R^+$ and so $I \subseteq R^+H$,
it follows from a Noether isomorphism theorem that $Q' \cong Q$ as $H$-modules.  Since $(H/R^+H)J = I/R^+H = 0$, it follows
from the lemma above that $d(Q, \M_R) = d(Q, \M_{R/J})$.  From Eq.~(\ref{eq: inequalityfordepth})
it follows that $$2d(Q,\M_R) + 1 \leq d(R,H), d(R/J, H/I) \leq 2d(Q, \M_R) + 2. $$ The proposition now 
follows from Theorem~\ref{prop-sigma}. 

If $d(R/J,H/I)$ is even, then $d(R/J,H/I) = 2 d(Q, \M_{R/J}) + 2$ from Eq.~(\ref{eq: inequalityfordepth}), but the Theorem and Eq.~(\ref{eq: inequalityfordepth}) imply that
$d(R,H) = d(R/J, H/I)$.  
\end{proof}

\begin{example}
\label{example-8dim}
\begin{rm}
Let $H_8$ be the $8$-dimensional small quantum group (at the root-of-unity $q = i$).  This is generated as an algebra by $K,E,F$ such that $K^2 = 1$, $E^2 = 0 = F^2$,
$EF = FE$, $FK = - KF$, and $EK = - KE$.  Let $R_4$ be the $4$-dimensional Hopf subalgebra
generated by $K,F$, which is isomorphic to the Sweedler algebra. (Please refer to \cite[Example 4.9]{K2013} for the coalgebra structure and details related to depth.)  
Consider the relatively nice Hopf ideal $I$ with basis $\{ F, FK, EF, EFK \}$.  Then $J$ is $\mbox{rad}\, R$ with basis $\{ F, FK \}$.  The quotient Hopf algebras $H_8/I \cong R_4$
and $R_4/J \cong \C \Z_2$ have depth $d( \C \Z_2, R_4) = 3$ computed in \cite[4.1]{Y}.  It follows from the proposition that
\begin{equation}
3 \leq d(R_4, H_8) \leq 4.
\end{equation}
\end{rm}
\end{example}

Recall that the core $\mbox{\rm Core}_G(H)$ of a subgroup pair of finite groups $H \leq G$ is the intersection of conjugate subgroups of $H$; equivalently, the largest normal subgroup contained in $H$.
Note that if $N = \mbox{\rm Core}_G(H)$, then $\mbox{\rm Core}_{G/N}(H/N)$ is the one-element group, i.e., $H/N \leq G/N$ is a \textit{corefree} subgroup. 
\begin{cor}
\label{cor-cores}
Suppose $N$ is a normal subgroup of a finite group $G$ contained in a subgroup $H \leq G$. For any ground field $k$, ordinary depth satisfies the inequality,
$$ d_k(H/N, G/N) \leq d_k(H,G) \leq d_k(H/N, G/N) + 1. $$
Moreover, if $d_k(H/N, G/N)$ is even, the two minimum depths are equal.  
\end{cor}
\begin{proof}
  Note that $I = kGkN^+$ is a relatively nice Hopf ideal in $kG$ (generated by $\{ g - gn \| \, g \in G, n \in N \}$). (And  conversely any Hopf ideal of the group algebra $kG$ is of this form \cite{PQ}.) 
Also note $kG/I \cong k[G/N]$.  The corollary
now follows from the previous proposition. 
\end{proof}

This result  improves \cite[Prop.\ 3.5]{HKY} to arbitrary characteristic.  Note too that combinatorial depth satisfies a completely analogous property  given in 
\cite[Theorem 3.12(d)]{BDK}. The inequality is in a sense the best
result obtainable.  
For example, let $G = \Sigma_4$, $H = D_8$, the dihedral group of
$8$ elements embedded in $G$, and $$N = \{ (1), (12)(34), (13)(24), (14)(23) \}.$$ It is noted after \cite[6.8]{BKK} that $G/N \cong \Sigma_3$,
$H/N \cong \Sigma_2$, and that minimum depth  satisfies  $d_0(G,H) = 4$ \cite[Example 2.6]{BKK}, $d_0(H/N, G/N) = 3$.  


\section{Depth of corings}
\label{two}
  Let $A$ be a ring and $(\mathcal{C}, A, \cop, \eps)$ be an $A$-coring.  Recall that
$\mathcal{C}$ is an $A$-bimodule with coassociative coproduct $\cop: \mathcal{C} \rightarrow
\mathcal{C} \otimes_A \mathcal{C}$ and counit $\eps: \mathcal{C} \rightarrow A$, both $A$-bimodule morphisms satisfying $(\eps \otimes \id_{ \mathcal{C}})\cop = \id_{ \mathcal{C}} = (\id_{ \mathcal{C}} \otimes \eps) \cop$ \cite{BW}.  It follows
that $\mathcal{C}^{\otimes_A n} \| \mathcal{C}^{\otimes_A (n+1)}$ as $A$-bimodules for each $n \geq 1$.  For convenience in notation let $\mathcal{C}^{\otimes_A 0} = A$.  The reader may consult \cite{BW} for more on corings, their morphisms, their comodules and  useful examples, such as coalgebras over ground rings and others we bring up below.  

\begin{definition}
An $A$-coring $\mathcal{C}$ has \textbf{depth} $2n+1$ if $\mathcal{C}^{\otimes_A n} \sim \mathcal{C}^{\otimes_A (n+1)} $, where $n$ is a nonnegative integer.  If $n$ is a positive integer,
$\mathcal{C}$ has depth $2n+1$ if $\mathcal{C}^{\otimes_A (n+1)} \| m \cdot \mathcal{C}^{\otimes_A n}$ for some $m \in \N$.  Note that $\mathcal{C}$ has depth
$2n+3$ if it has depth $2n+1$.   If $\mathcal{C}$ has a finite depth, let $d(\mathcal{C}, A)$ denote
its \textbf{minimum depth}.  
  \end{definition}

\begin{example}
\begin{rm}
Given a ring extension $B \rightarrow A$, let $\mathcal{C}$ denote its Sweedler $A$-coring
$A \otimes_B A$ with coproduct $\cop: A \otimes_B A \rightarrow A \otimes_B A \otimes_B A$
($= A^{\otimes_B 3}$) given by $\cop(a \otimes_B a') = a \otimes 1 \otimes a'$
and counit $\mu: A \otimes_B A \rightarrow A$ by $\mu(a \otimes a') = aa'$ \cite{BW}.  
Note that $\mathcal{C}^{\otimes_A n} \cong A^{\otimes_B (n+1)}$ for  integers $n \geq 0$.  
 Therefore, comparing the tensor powers of $\mathcal{C}$ as natural $A$-bimodules is equivalent
to comparing the tensor powers $A^{\otimes_B n}$ as in the definition of h-depth in Definition~\ref{def-depth}.  It follows that
\begin{equation}
\label{eq: coringdepth=hdepth}
d(A \otimes_B A, A) = d_h(B,A)
\end{equation}
\end{rm}
\end{example}
An $A$-coring $\mathcal{C}$ has a \textit{grouplike} element $g \in \mathcal{C}$ if $\cop(g) = g \otimes_A g$
and $\eps(g) = 1_A$.  For example, the Sweedler $A$-coring $A \otimes_B A$ in the example
has a grouplike element $1_A \otimes_B 1_A$.  An $A$-coring $\mathcal{C}$ with grouplike $g$
has invariant subalgebra $A_g^{\mbox{co}\, \mathcal{C}} := \{ b \in A \| \, bg = gb \} := B$ (notation refers to coinvariants of $\mathcal{C}$-comodule $A$ determined by grouplike \cite[p.\ 278]{BW}).  Recall that
$\mathcal{C}$ is a \textit{Galois coring} if $A \otimes_B A \cong  \mathcal{C}$ via $a \otimes_B a' \mapsto
aga'$, a coring isomorphism of $\mathcal{C}$ with the Sweedler coring of $B \subseteq A$.  
For example, the Sweedler coring of a faithfully flat ring extension $A \supseteq B$ is Galois, since $B =
\{a \in A \| \,  a \otimes_B 1_A = 1_A \otimes_B a \}$ follows from $A_B$ or ${}_BA$ being
faithfully flat (\cite[28.6]{BW}, or see Lemma~\ref{lemma-ff}).  
\begin{example}
\label{example-Frobenius}
\begin{rm}
Recall that $A \supseteq B$ is a Frobenius extension  with $F: A \rightarrow B$ a ``Frobenius'' homomorphism and $e := \sum_i x_i \otimes_B y_i$ a ``dual bases tensor,'' satisfying  $ae= ea$ for every $a \in A$, if $A$
is a $B$-coring with coproduct $\cop:A \rightarrow A \otimes_B A, \ \cop(a) = ae$ and counit $F$, this coring being denoted by
$\mathcal{C}_{\mathrm{Frob}}$.  The more familiar conditions of the counit equations
characterize Frobenius extensions \cite{NEFE}. 
   The tensor powers of this coring are now given
by natural $B$-bimodules $\mathcal{C}_{\mathrm{Frob}}^{\otimes_B n} = A^{\otimes_B n}$.
Using Definition~\ref{def-depth} for the definition of odd depth, we obtain
\begin{equation}
\label{eq: FrobeniusCoringDepth=OddDepth}
d(\mathcal{C}_{\mathrm{Frob}}, B) = d_{\mathrm{odd}}(B,A)
\end{equation}
in terms of the minimum odd depth.  
\end{rm}
\end{example}  
\section{Depth of Coalgebra-Galois Extensions}
\label{three}

In this section we define depth of a certain coalgebra-Galois extension and see that its minimum depth takes on at least as many interesting values as subgroup depth \cite{BK, BDK, BKK, D, F, FKR}.  In contrast, the minimum depth of a Hopf-Galois extension is one or two.    

Let $k$ be a field; all unlabeled tensors are over $k$. We begin with a review of the entwining
structure of an  algebra $A$ and  coalgebra $C$.  The entwining (linear) mapping $\psi: C \otimes A \rightarrow A \otimes C$ satisfies two 
commutative pentagons and two triangles (a bow-tie diagram on \cite[p.\ 324]{BW}).  Equivalently,
$(A \otimes C, \id_A \otimes \cop_C, \id_A \otimes \eps_C)$ is an $A$-coring with respect
to the $A$-bimodule structure $a(a' \otimes c){a''} = aa' \psi(c \otimes a'')$
(or conversely defining $\psi(c \otimes a) = (1_A \otimes c) a$).

An entwining structure mapping $\psi: C \otimes A \rightarrow A \otimes C$ takes
values that may be denoted by $\psi(c \otimes a) = a_{\alpha} \otimes c^{\alpha} = a_{\beta} \otimes c^{\beta}$, suppressing linear sums of rank one tensors, and satisfies the axioms: (for all $a,b \in A, c \in C$)
\begin{enumerate}
\item  $\psi(c \otimes ab) = a_{\alpha}b_{\beta} \otimes {c^{\alpha}}^{\beta}$;
\item $\psi(c \otimes 1_A) = 1_A \otimes c$;
\item $a_{\alpha} \otimes \cop_C(c^{\alpha}) = {a_{\alpha}}_{\beta} \otimes {c\1}^{\beta} \otimes {c\2}^{\alpha}$
\item $a_{\alpha} \eps_C(c^{\alpha}) = a \eps_C(c)$,
\end{enumerate}
which is equivalent to two commutative pentagons (for axioms 1 and 3) and two commutative
triangles (for axioms 2 and 4), in an exercise.

The following is  \cite[32.6]{BW} or \cite[Theorem 2.8.1]{CMZ}. 

\begin{prop}
\label{prop-1-1}
Entwining structures $\psi: C \otimes A \rightarrow A \otimes C$ are in one-to-one correspondence with $A$-coring structures \newline
 $(A \otimes C, \id_A \otimes \cop_C, \id_A \otimes \eps_C)$.
\end{prop}
\begin{proof}
Given an entwining $\psi$, the obvious structure maps $(A \otimes C, \id_A \otimes \cop_C, \id_A \otimes \eps_C)$ form an $A$-coring with respect
to the $A$-bimodule structure $a(a' \otimes c){a''} = aa' \psi(c \otimes a'')$.  
Conversely, given an $A$-coring $A \otimes C$ with
coproduct $\id_A \otimes \cop_C: A \otimes C \rightarrow A \otimes C \otimes C \cong
A \otimes C \otimes_A A \otimes C$ and counit $\id_A \otimes \eps_C: A \otimes C \rightarrow
A \otimes k \cong A$, one defines $\psi(c \otimes a) = (1_A \otimes c) a$ and checks
that $\psi$ is an entwining, and the other details, in an exercise.
\end{proof}

Our primary example in this section is $A = H$, a Hopf algebra with coproduct $\cop$, counit $\eps$
and antipode $S: H \rightarrow H$, and $C$ a right $H$-module coalgebra,
i.e. a coalgebra $(C, \cop_C, \eps_C)$ and module $C_H$ satisfying $\cop_C(ch) = c\1 h\1 \otimes c\2 h\2$ 
and $\eps_C (ch) = \eps_C(c) \eps(h)$ for each $c \in C, h \in H$. An entwining mapping $\psi: C \otimes H \rightarrow  H \otimes C$ is defined by $\psi(c \otimes h) = h\1 \otimes c h\2$.
The entwining axioms are checked in a more general setup in Section~\ref{four}.

The associated $H$-coring $H \otimes C$ has coproduct $\id_H \otimes \cop_C$ and counit
$\id_H \otimes \eps_C$ with $H$-bimodule structure: ($x,y,h \in H, c \in C$) 
\begin{equation}
\label{eq: coring}
x(h \otimes c) y = xh y\1 \otimes c y\2
\end{equation}
Notice that this is the diagonal action from the right.  

\begin{prop}
\label{prop-entwinedepth}
The depth of the $H$-coring $H \otimes C$ and the depth of the $H$-module $C$ are related by
$d(H \otimes C, H) = 2d(C, \mathcal{M}_H) + 1$.  
\end{prop}

\begin{proof}
The $n$-fold tensor product of $H \otimes C$ over $H$ reduces  to the $H$-bimodule isomorphism 
\begin{equation}
\label{eq: iso}
(H \otimes C)^{\otimes_H n} \cong H \otimes C^{\otimes n}
 \end{equation}
via the mapping 
$$ \otimes_{i=1}^n h_i \otimes c_i \longmapsto h_1 {h_2}\1 \cdots {h_n}\1 \otimes_{i=1}^{n-1} c_i {h_{i+1}}\i+1 \cdots {h_n}\i+1 \otimes c_n $$
with inverse $h \otimes c_1 \otimes \cdots \otimes c_n \mapsto h \otimes c_1 \otimes_H 1_H \otimes c_2 \otimes_H \cdots \otimes_H 1_H \otimes c_n$, 
where $H \otimes C^{\otimes n}$ has right $H$-module structure from the diagonal action by $H$:
$(h \otimes c_1 \otimes \cdots \otimes c_n)x = hx\1 \otimes c_1 x\2 \otimes \cdots \otimes c_n x_{(n+1)}$.  This follows  from Eq.~(\ref{eq: coring}), cancellations of the type $M \otimes_H H \cong M$ for modules $M_H$, and an induction on $n$.  

Suppose $d(C, \mathcal{M}_H) = n$, so that $C^{\otimes n} \sim C^{\otimes (n+1)}$ as right
$H$-modules (in the finite tensor category $\mathcal{M}_H$).  Applying an additive functor,
it follows that $H \otimes C^{\otimes n} \sim H \otimes C^{\otimes (n+1)}$ as $H$-bimodules.
Applying the isomorphism~(\ref{eq: iso})  the coring depth satisfies $d(H \otimes C, H) \leq 2d(C, \mathcal{M}_H) + 1$.

Conversely, if $d(H \otimes C, H) = 2n+1$, so that $H \otimes C^{\otimes n} \sim H \otimes C^{\otimes (n+1)}$ from Eq.~(\ref{eq: iso}) again, we apply that additive functor $k \otimes_H -$ to the similarity
and obtain the similarity of right $H$-modules, $C^{\otimes n} \sim C^{\otimes (n+1)}$.  Thus $2d(C, \mathcal{M}_H) +1 \leq d(H \otimes C, H)$.   
\end{proof}

Now suppose $R \subseteq H$ is a left coideal subalgebra of a finite-dimensional Hopf algebra;
i.e., $\cop(R)\subseteq  H \otimes R$.  Let $R^+$ denote the kernel of the counit restricted to $R$.
Then $R^+H$ is a right $H$-submodule of $H$ and a coideal by a short computation given
in \cite[34.2]{BW}.  Thus $Q := H/R^+H$ is a right $H$-module coalgebra (with a right $H$-module
coalgebra epimorphism $H \rightarrow Q$ given by $h \mapsto h + R^+H := \overline{h}$).
The $H$-coring $H \otimes Q$ has grouplike element $1_H \otimes \overline{1_H}$; in fact,
\cite[34.2]{BW} together with \cite{S} shows that this coring is Galois:
\begin{equation}
\label{eq: map}
H \otimes_R H \stackrel{\cong}{\longrightarrow} H \otimes Q
\end{equation}
via $x \otimes_R y \mapsto x y\1 \otimes \overline{y\2}$ (also noted in \cite{HKY} and in \cite{U}
for Hopf subalgebras).  That $H_R$ is faithfully flat follows from the result that $R$ is a Frobenius algebra and  $H_R$ is free \cite{S}.  Note that an inverse to (\ref{eq: map}) is given
by $x \otimes \overline{z} \mapsto xS(z\1) \otimes_R z\2$ for all $x,z \in H$. 

From Proposition~\ref{prop-entwinedepth}, Eqs.~(\ref{eq: map}) and~(\ref{eq: coringdepth=hdepth}) we note the following.    

\begin{cor}
\label{cor-cue}
For a left coideal subalgebra $R$ in a  Hopf algebra $H$, its h-depth
is related to the module depth of $Q$ by 
\begin{equation}
\label{eq: cue}
d_h(R,H) = 2d(Q, \mathcal{M}_H) + 1.
\end{equation}   
\end{cor}
\begin{proof}
The proof is sketched above for finite-dimensional $H$.  For infinite-dimensional $H$, note
that the $H$-bimodule isomorphism in Eq.~(\ref{eq: map}) remains valid, as does Proposition~\ref{prop-entwinedepth} and Eq.~(\ref{eq: coringdepth=hdepth}).
\end{proof}
Suppose $R$ is a Hopf subalgebra of $H$. Then $Q$ is an $R$-module coalgebra by restriction.
  A similar argument to the one above shows that $d(R,H)$ and $d(Q, \mathcal{M}_R)$ 
satisfy the inequalities in (\ref{eq: inequalityfordepth}).  

Finally we recall that a \textit{$C$-Galois extension} $A \supseteq B$, where $C$ is a coalgebra and
$A$ a right $C$-comodule via coaction $\delta: A \rightarrow A \otimes C$,  the subalgebra of coinvariants is characterized by satisfying $B =
\{ b \in A \| \, \forall a \in A, \delta(ba) = b\delta(a) \}$ and $\beta: A \otimes_B A \mapsto
A \otimes C$ given by $\beta(a \otimes a') = a \delta(a')$ is bijective.  For example,
a left coideal subalgebra $R \subseteq H$ is a coalgebra-Galois extension with respect to 
the ($H$-module) coalgebra $Q$, as sketched above (the details are in \cite[34.2]{BW}).  
Of course, this applies to Hopf subalgebras and more particularly to finite group algebra extensions.
Then we see that coalgebra-Galois extensions have at least the range of values computed for
subgroup depth \cite{BDK, BKK, D, F, FKR, HHP}.

\subsection{A faithfully flat interlude} Let $C$ be a coalgebra and $H$-module quotient of $H$, with canonical epi $H \rightarrow C$ of right $H$-module coalgebras, and $A$ is a right $H$-comodule algebra, with the obvious $C$-comodule coaction $\delta: A \rightarrow A \otimes C$ given by $\delta(a) = a\0 \otimes \overline{a\1}$.  Define the subalgebra $B = \{ b \in A \, | \, \delta(b) = b \otimes \overline{1_H} \}$.  Now suppose that $D \subseteq B$ is a subalgebra for which the canonical (Galois) mapping $\beta: A \otimes_D A \stackrel{\cong}{\longrightarrow} A \otimes C$
given by  $a \otimes_D a' \mapsto a\delta(a') = a{a'}\0 \otimes \overline{ {a'}\1}$, is an isomorphism.
If either of the natural modules $A_D$ or ${}_DA$ is faithfully flat, then $D = B$.  
This follows from noting that for each $b \in B$,  $\beta(1_A \otimes_D b) = b \otimes \overline{1_H} = \beta(b \otimes_D 1_A)$ and the next lemma. 
\begin{lemma}
\label{lemma-ff}
If $b \otimes_D 1 = 1 \otimes_D b$ for $b \in A \supseteq D$ with $D$ a subring of $A$ such that
$A_D$ or ${}_DA$ is faithfully flat, then $b \in D$.
\end{lemma}
\begin{proof}
Recall that a flat module $A_D$ is faithfully flat iff for each module ${}_DN$ such that $A \otimes_D N = 0$,
it follows that $N = 0$.  Now form the module $N = B'/D$, where $B' = \{ b \in A \, | \,  b \otimes_D 1_A = 1_A \otimes b \}$.  By an exercise with a commutative square, $A \otimes_D N = 0$, whence $D = B'$.  
\end{proof}

\subsection{A normal element characterization of left ad-stabity for left coideal subalgebras}
Suppose $R$ is left coideal subalgebra of a finite-dimensional Hopf algebra $H$.  Let $Q$ be its quotient
right $H$-module coalgebra defined above.  By Skryabin's Freeness Theorem \cite{S} $\dim R$ divides $\dim H$,
and $(R, \eps)$ is an augmented Frobenius algebra, so $R$ has a nonzero right integral $t_R$ unique up to scalar multiplication \cite[Ch. 6]{NEFE}.  Then one proves just as in \cite[Lemma 3.2]{K2013} that
\begin{equation}
\label{isomorph}
Q \stackrel{\cong}{\longrightarrow} t_R H
\end{equation}
via $h + R^+H \mapsto t_R h$.  

Recall that a subalgebra $R$ in a Hopf algebra $H$ is stable under the left adjoint action of $H$ 
if $h\1 r S(h\2) \in R$ for all $r \in R$, $h \in H$.  One briefly says that $R$ is left ad-stable
(or left normal \cite{B2012}).  In section~3 of \cite{B2012}, Burciu shows how to prove the following.

\begin{lemma}(cf.\ \cite{B2012})
\label{lem-straightenedout}
 A left coideal subalgebra $R$ in a finite-dimensional Hopf algebra $H$ is left ad-stable if and only if  $HR^+ \subseteq  R^+H$ if and only if the subalgebra $R \subseteq H$ has right depth $2$.
\end{lemma}
\begin{proof} For the convenience of the reader, we sketch the proof.  If $R$ is left ad-stable in $H$,
$r \in R^+$ and $h \in H$, then $hr = h\1 r S(h\2) h\3 \in R^+H$.  

If $HR^+ \subseteq R^+H$, the isomorphism $\beta: H \otimes_R H \stackrel{\cong}{\rightarrow}
H \otimes Q$ in Eq.~(\ref{eq: map}) is an $H$-$R$-bimodule morphism, since $Q_R$ is trivial.
Hence ${}_H H \otimes_R H_R \cong (\dim Q) \cdot {}_H H_R$ and $R$ has right depth $2$ in $H$.

If ${}_H H \otimes_R H \oplus * \cong n \cdot {}_H H_R$, then applying $k \otimes_H -$ yields
$Q_R \oplus * \cong n \cdot k_R$.  Since $R^+ = \Ann\, k_R$, it follows easily that $HR^+ \subseteq R^+H$.

If $HR^+ \subseteq R^+H$, then $R^+ H$ is a bi-ideal in $H$,  $Q$ is a bialgebra and has a natural $H$-bimodule structure  from multiplication in $H$.  Also $\beta$ defined above satisfies $\beta(1_H \otimes_R h) = 
\beta(h \otimes_R 1_H)$ for $$h \in H^{\mbox{co}\, Q} = \{ x \in H | x\1 \otimes \overline{x\2} = x \otimes \overline{1} \}. $$
Of course, $R \subseteq H^{\mbox{co}\, Q}$.  
Since $\beta$ is bijective, $h \otimes_R  1 = 1 \otimes_R h$, thus by Lemma~\ref{lemma-ff},
$R = H^{\mbox{co}\, Q}$.  But $H^{\mbox{co}\, Q}$ is  left ad-stable by an argument used in
\cite[3.4.2]{M} modified slightly by the remark in the first sentence of this paragraph.  
\end{proof}

The following generalizes \cite[5.2]{K2013}.
\begin{theorem}
\label{th-leftrightleft}
A left coideal subalgebra $R$ in a finite-dimensional Hopf algebra $H$ is left ad-stable if and only if its right integral $t_R$ is a normal element in $H$.  
\end{theorem}
\begin{proof}
($\Rightarrow$)  One argues as in \cite[5.2]{K2013}, using Eq.~(\ref{isomorph}) and $Q$ is under the hypothesis a trivial right $R$-module, that  $h\1 t_R Sh\2$
is a nonzero right integral in $R$ for any $h \in H$, so that $$ht_R = h\1 t_R S(h\2) h_3 \in t_R H.$$
Hence, $Ht_R \subseteq t_R H$.  For the opposite inclusion,  note that $$t_R S(h) = S(h\1) h\2 t_R S(h\3) \in Ht_R.$$  

($\Leftarrow$) If $Ht_R = t_R H$, then $\Ann \, Q_R \supseteq R^+$.  Thus $HR^+ \subseteq  R^+H$.
We conclude that $R$ is left ad-stable in $H$ from Lemma~\ref{lem-straightenedout}.  
\end{proof}
One says that a subalgebra $R$ of a Hopf algebra $H$ is right ad-stable if $S(h\1)rh\2 \in R$
for all $h \in H, r \in R$. 
Applying the antipode anti-automorphism to the above theorem, one obtains (left as an exercise) the correct formulation and proof of the opposite (and equivalent) theorem.

\begin{cor}
\label{rightversion}
A right coideal
subalgebra is right ad-stable in a finite-dimensional Hopf algebra $H$ iff its left integral $t_L$ is a normal element in $H$.  
\end{cor} 
\begin{example}
\begin{rm}
Consider the $8$-dimensional small quantum group $H = H_8$ given as an algebra in Example~\ref{example-8dim}, with coalgebra
structure given by $\cop(K) = K \otimes K$, $\cop(E) = E \otimes 1 + K \otimes E$, and $\cop(F) = F \otimes K + 1 \otimes F$.  It follows that $S(K) = K$, $S(E) = -KE$ and $S(F) = - FK$.  

Consider the $2$-dimensional Frobenius subalgebra $R$ generated by $E$ (the ``ring of dual numbers'').  Notice that $R$ is a left coideal subalgebra, since $\cop(E) \in H \otimes R$.  
It is left ad-stable since $FEK + ES(F) = 0$.  Also $t_R = E$ ($= t_L$) and $HE = EH$ is the $4$-dimensional
vector subspace of $H$ with basis $\{ E, EF, EK, EFK \} $: equivalently, $HR^+ = R^+ H$.  

However, $R$ is not right ad-stable, since $S(F)EK + EF = 2EF \not \in R$.  The Frobenius subalgebra
of dimension $2$ generated by $F$ in $H$ is however a right coideal subalgebra that is right ad-stable
with normal integral element, and provides an example of Corollary~\ref{rightversion}.   This example does not
contradict the correct  formulation of the opposite of Lemma~\ref{lem-straightenedout}: 
\end{rm}
\end{example} 
\begin{prop}[cf.\ \cite{B2012}]
A right coideal subalgebra $R$ of $H$ 
is right ad-stable iff $R^+H \subseteq HR^+$ iff $R$ has left depth $2$ in $H$. 
\end{prop} 
\section{The coring of an entwining structure of a comodule algebra and a module coalgebra}
\label{four}

Let $H$ be a Hopf algebra.  Suppose $A$ is a right $H$-comodule algebra, i.e., there is a coaction
$\rho: A \rightarrow A \otimes H$, denoted by $\rho(a) = a\0 \otimes a\1$ that is an algebra homomorphism and $(A, \rho)$ is a right $H$-comodule \cite{BW, CMZ, M}.  Moreover, let
$(C, \cop_C, \eps_C)$ be a right $H$-module coalgebra, i.e., a coalgebra in the tensor category
$\mathcal{M}_H$ introduced in more detail in Section~\ref{three}.  

\begin{example}
\begin{rm}
The Hopf algebra $H$ is right $H$-comodule algebra over itself, where $\rho = \cop$.  Given a Hopf subalgebra
$R \subseteq H$ the quotient module $Q$ defined in Section~1 as $Q = H/ R^+H$ is a right $H$-module
coalgebra. 

Note that $(H, \cop, \eps)$
is also a right $H$-module coalgebra.  The canonical epimorphism $H \rightarrow Q$ denoted by $h \mapsto \overline{h}$ is an epi of right $H$-module coalgebras, and module $Q_H$ has $\overline{1_H}$ as a cyclic generator.

Of course, if $H = k$ is the trivial one-dimensional Hopf algebra, $A$ may be any $k$-algebra and $C$
any $k$-coalgebra.  
\end{rm}
\end{example}   

The mapping $\psi: C \otimes A \rightarrow A \otimes C$ defined by 
$\psi(c \otimes a) = a\0 \otimes c a\1$ is an entwining as the reader may easily check (the so-called Doi-Koppinen entwining
\cite[33.4]{BW}, \cite[2.1]{CMZ}, which includes the case considered in Section~\ref{three}).

From Proposition~\ref{prop-1-1} it follows that $A \otimes C$ has $A$-coring structure
\begin{equation}
\label{eq: comodulealgdiagonal}
a(a' \otimes c)a'' = aa' {a''}\0 \otimes c {a''}\1
\end{equation}
which defines the bimodule ${}_A(A \otimes C)_A$.  The coproduct is given by $\id_A \otimes \cop_C$
and the counit by $\id_A \otimes \eps_C$.   

Suppose in this paragraph that $H$ is a finite-dimensional Hopf algebra over an algebraically closed field, in which case $\mathcal{M}_H$ 
is a finite tensor category \cite{EO}. 
Notice in the equation above that the right $A$-module is given by a version of the diagonal action
in which the category $\mathcal{M}_A$ is a module category over $\mathcal{M}_H$ \cite{AM, EO}.
 
\begin{prop}
\label{prop-comodalgentwinedepth}
The depth of the $A$-coring $A \otimes C$ and the depth of the $H$-module $C$ are related by
$d(A \otimes C, A) \leq  2d(C, \mathcal{M}_H) + 1$.  
\end{prop}
\begin{proof}
The $n$-fold tensor product of $A \otimes C$ over $A$ reduces  to the $A$-bimodule isomorphism 
\begin{equation}
\label{eq: iso2}
(A \otimes C)^{\otimes_A n} \cong A \otimes C^{\otimes n}
 \end{equation}
via the mapping $$a \otimes c_1 \otimes \cdots \otimes c_n \mapsto a \otimes c_1 \otimes_A 1_A \otimes c_2 \otimes_A \cdots \otimes_A 1_A \otimes c_n,$$
where $A \otimes C^{\otimes n}$ has right $A$-module structure from the diagonal action by $A$:
$(a \otimes c_1 \otimes \cdots \otimes c_n)b = ab\0 \otimes c_1 b\1 \otimes \cdots \otimes c_n b_{(n)}$.  This follows  from Eq.~(\ref{eq: comodulealgdiagonal}), cancellations of the type $M \otimes_A A \cong M$ for modules $M_A$, and an induction on $n$.  

Suppose $d(C, \mathcal{M}_H) = n$, so that $C^{\otimes n} \sim C^{\otimes (n+1)}$ as right
$H$-modules (in the finite tensor category $\mathcal{M}_H$).  Applying an additive functor,
it follows that $A \otimes C^{\otimes n} \sim A \otimes C^{\otimes (n+1)}$ as $A$-bimodules.
Then applying the isomorphism~(\ref{eq: iso2}) obtain $d(A \otimes C, A) \leq 2d(C, \mathcal{M}_H) + 1$.
\end{proof}

Suppose $C$ has grouplike $g'$.  Then $g = 1_A \otimes g'$ is a grouplike
element in $A \otimes C$.  Then we have from Eq.~(\ref{eq: coringdepth=hdepth}) and Proposition~\ref{prop-comodalgentwinedepth} the proof of the following. 
\begin{lemma}
\label{lem-ineq}
Suppose $A \otimes C := \mathcal{C}$ is a Galois coring with $B = A_g^{\mbox{co}\, \mathcal{C}}$
(or just that $B$ is a subalgebra of $A$ such that $A \otimes_B A \cong \mathcal{C}$ as
natural and right-diagonal $A$-bimodules, respectively).  
 Then the h-depth of the subalgebra $B \subseteq A$ satisfies the inequality, $d_h(B,A) \leq 2d(C, \mathcal{M}_H) +1$.  
\end{lemma}

\subsection{Crossed products as comodule algebras}
Now let $A$ be an associative crossed product $D \#_{\sigma} H$ for some Hopf algebra algebra $H$ and twisted $H$-module algebra $D$ with $2$-cocycle $\sigma: D \otimes D \rightarrow H$ that
is convolution-invertible \cite[chapter 7]{M}.  Then $A$ is a right $H$-comodule algebra via
$\rho = \id_D \otimes \cop$, which is quite obvious from the formula for multiplication in   
$D \#_{\sigma} H$: ($h,x,y \in H, d,d' \in D$) 
\begin{equation}
\label{eq: crossprod}
(d \# h)(d' \# x) = d(h\1 \cdot d')\sigma(h\2, x\1) \# h\3 x\2
\end{equation}
Additionally, the formulas for $D$ a twisted right $H$-module and $\sigma$ a $2$-cocycle are useful:
\begin{equation}
\label{eq: twistedmodule}
h \cdot (x \cdot d) = \sigma(h\1, x\1) (h\2 x\2 \cdot d) \sigma^{-1}(h\3, x\3)
\end{equation}
\begin{equation}
\label{eq: twococycle}
(h\1 \cdot \sigma(x\1, y\1)) \sigma( h\2, x\2 y\2) = \sigma(h\1, x\1) \sigma( h\2 x\2, y)
\end{equation}

\begin{example}
\begin{rm}
For $H = kG$ a group algebra, one obtains from this the familiar conditions of the group crossed product, $A = D \# kG$, 
where for $g,h, s \in G$ and $d, d' \in D$, 
\begin{eqnarray}
g \cdot (h \cdot d) & =  & \sigma(g, h) (gh \cdot d)\sigma(g,h)^{-1} \label{action} \\
(g \cdot \sigma(h,s)) \sigma(g, hs) & =  &  \sigma(g,h) \sigma(gh, s) \label{cocycle}\\
(d \# g) (d' \# h) & = & d(g \cdot d') \sigma(g,h) \# gh 
\end{eqnarray} 
Note that $\sigma^{-1}(g,h)$ is interchangeable with $\sigma(g,h)^{-1}$.
\end{rm}
\end{example}
\begin{example}
\begin{rm}
If $\sigma$ is trivial, $\sigma = 1_D$, then the crossed product reduces to the skew group algebra
$D \ast G$, where $D$ is a $G$-module and multiplication is given by $(d \# g)(d' \# h) = d (g \cdot d') \# gh$; this is a well-known setup  in Galois theory of fields and commutative algebras.     More generally,
$D \#_{\sigma} H$ is clearly isomorphic  to the smash product $D \# H$ if $\sigma(x,y) = \eps(x)\eps(y)1_H$:
see Eq.~(\ref{eq: crossprod}).  Then $D$ is a left $H$-module algebra. (If the action is Galois,
then the endomorphism algebra of $D$ over its invariant subalgebra is isomorphic to the smash
product $D \# H$ \cite[8.3.3]{M}.)
\end{rm}
\end{example}

\begin{example}
\begin{rm}
If $\sigma$ is instead nontrivial, but the action of $G$ on $D$ is trivial, i.e. $g \cdot d = d$
for each $g \in G$ and $d \in D$, then $D \#_{\sigma} kG = D_{\sigma}[G]$, the twisted group
algebra \cite{P}, with multiplication given by $(d \# g)(d' \# h) = dd' \sigma(g,h) \# gh$.  For example,
the real quaternions are a twisted group algebra of $\R$ with $G \cong \Z_2 \oplus \Z_2$ and
$\sigma = \pm 1$.  
\end{rm}
\end{example}

\begin{example}
\label{example-groupascrossedproduct}
\begin{rm}
For any group $G$ and normal subgroup $N$ of $G$, the group algebra is a crossed product of the quotient group algebra acting on
the subgroup $N$ as follows \cite[7.1.6]{M}. Let $Q$ denote the group $G/N$.  For each coset $q \in Q$, let $\gamma(q)$ denote a coset representative, choosing $\gamma(\overline{1_G}) = 1_G$.  It
is an exercise then to show that with $\sigma(\overline{x},\overline{y}) := \gamma(\overline{x}) \gamma(\overline{y}) \gamma(\overline{x}\overline{y})^{-1} \in N$ and action
of $Q$ on $kN$ given by $\overline{x}\cdot n = \gamma(\overline{x})n \gamma(\overline{x})^{-1}$, one has
$kG = kN \#_{\sigma} kQ$.  
\end{rm}
\end{example}
 
Let $R \subseteq H$ be a Hopf subalgebra pair.  Again form the quotient module $Q = H / R^+H$,
which is a right $H$-module coalgebra with canonical epi of right $H$-module coalgebras
$\pi: H \rightarrow Q$, $h \mapsto \overline{h}$.  Then the crossed product 
$A = D \#_{\sigma} H$, which is a right $H$-comodule algebra, is also a right $Q$-comodule
via $A \rightarrow A \otimes H \stackrel{A \otimes \pi}{\longrightarrow} A \otimes Q$, the composition $\rho$ being the coaction given by $\rho(d \# h) = d \# h\1 \otimes \overline{h\2}$.
Then one sees that $B : = D \#_{\sigma} R \subseteq A^{{\mbox co}\, Q} $, the $Q$-coinvariants defined as the subalgebra $\{ b \in A \| \, \forall a \in A, \rho(ba) = b\rho(a) \}$; the inclusion follows from noting $\overline{rh} = \eps(r) \overline{h}$ for all $r \in R, h \in H$.   

Note  that $Q$ has the grouplike element $\overline{1_H}$, so that $g = 1_A \otimes \overline{1_H}$
is a grouplike element in the $A$-coring $\mathcal{C} = A \otimes Q$. From the previous observation and
Eq.~(\ref{eq: comodulealgdiagonal}), it follows that $B = D \# R \subseteq A_g^{\mbox{co}\, \mathcal{C}}$.  

Suppose
that $A \otimes_B A \cong A \otimes Q$ as $A$-bimodules, or more strongly $\mathcal{C}$
is a Galois $A$-coring with respect to $B = A_g^{\mbox{co}\, \mathcal{C}}$.  Then putting together Lemma~\ref{lem-ineq} with Corollary~\ref{cor-cue} and Eq.~(\ref{eq: cue}),  we arrive at the inequality of minimum
h-depths, 
\begin{equation}
\label{eq: mainresult}
d_h(D \#_{\sigma} R, D \#_{\sigma} H) \leq d_h(R, H),
\end{equation}
the main aim of this paper:  we next set about establishing this supposition for finite rank group algebra extensions and certain crossed product Hopf algebra extensions.  
\begin{prop}
\label{prop-sch}
Consider  the $H$-comodule algebra $A = D \#_{\sigma} H$.  Let $B = D \#_{\sigma} R$ be
an $R$-comodule subalgebra of $A$, where $R$ is a Hopf subalgebra of $H$.  Suppose  the condition~(\ref{eq: cond}) below is satisfied. Then $A \otimes_B A \cong A \otimes Q$ as $A$-bimodules,
where $Q = H/ R^+H$.  
 If $A_B$ is faithfully flat,  then $A$ is a coalgebra $Q$-Galois extension of $B$, i.e., 
$A \otimes Q$ is a Galois $A$-coring with $B$ the invariant subalgebra.  
\end{prop}
\begin{proof}
We investigate if the canonical mapping $$\beta(a \otimes_B a') = aga' = a(d \# h\1) \otimes \overline{h\2}$$
(where $a' = d \# h \in D \# H = A$) is bijective, and the module $A_B$ is  faithfully flat.  

The map $\beta: A \otimes_B A \stackrel{\cong}{\longrightarrow} A \otimes Q$ is well-defined: 
suppose $d' \in D, r \in R$, we  check that $$\beta(a(d' \# r) \otimes_B  (d \# h)) = \beta(a \otimes_B (d'(r\1 \cdot d)\sigma(r\2,h\1) \# r\3 h\2))$$
i.e., $ a(d' \# r)(d \# h\1) \otimes \overline{h\2} = ad'(r\1 \cdot d)\sigma(r\2,h\1) \# r\3 h\2 \otimes \overline{r\4 h\3}$, which is clear since $\overline{rh} = \eps(r) \overline{h}$ for each $r \in R, h \in H$.  

We note that $\beta^{-1}: A \otimes Q \rightarrow A \otimes_B A$ is given by
$$\beta^{-1}(a \otimes \overline{h}) = a(\sigma^{-1}(S(h\2), h\3) \# S(h\1)) \otimes_B (1_D \# h\4).$$
The computation $\beta^{-1} \circ \beta = \id_{A \otimes_B A}$ and the computation
$\beta \circ \beta^{-1} = \id_{A \otimes Q}$ follow from $\gamma: H \rightarrow A$,
defined by $\gamma(h) = 1_D \# h$, having convolution-inverse
$\mu: H \rightarrow A$, defined by $$\mu(h) = \sigma^{-1}(S(h\2),h\3) \# S(h\1),$$ 
on \cite[p.\ 109]{M}.  We are left with verifying that $\beta^{-1}$ is well-defined, i.e., vanishes when $h = rh'$
for some $r \in R^+, h' \in H$.  Dropping the prime, this becomes the condition
\begin{equation}
\label{eq: cond}
\sigma^{-1}(S(h\2) S(r\2), r\3 h\3) \# S(h\1) S(r\1) \otimes_B (1_D \# r\4 h\4) = 0
\end{equation}
for all $r \in R, h \in H$ such that $\eps(r) = 0$. 

Finally, suppose either natural module, $A_B$ or ${}_BA$ is faithfully flat.  Then $A$ is a $Q$-Galois extension of $B$ (cf.\ \cite[34.10]{BW}, \cite[1.2, 3.1]{Sch}).
\end{proof}
\begin{cor}
Let $R \subseteq H$ be a finite-dimensional Hopf subalgebra pair, and $A$ a left $H$-module algebra.  Then
h-depth satisfies 

\begin{equation}
\label{eq: smashproddepth}
d_h(A \# R, A \# H) \leq d_h(R,H).
\end{equation}  
\end{cor}
\begin{proof}
If $\sigma(x,y) = \eps(x)\eps(y)1_D$ in Proposition~\ref{prop-sch}, then  the crossed products in the theorem are smash products. The corollary follows if Eq.~(\ref{eq: cond}) is satisfied with this choice of  $\sigma$.  But  the left-hand side
reduces  to $1 \# S(h\1) \otimes_B 1 \# S(r\1) r\2 h\2 $, indeed equal to zero for $r \in R^+$.
Finally, the  natural module ${}_RH$ is free by Nichols-Zoeller, so that $A \# H$ is easily shown to be  free as a natural left module
over $A \# R$.  
\end{proof}
\begin{prop}
Suppose $G$ is a group and $S$ is a subgroup of $G$ such that $|G:S| < \infty$.  Let $H = kG$, $R = kS$ and $A$,$B$
be crossed products of the group algebras $H$ and $R$ with left twisted $H$-module algebra $D$ as above.  Then $A$ is a $Q$-Galois extension of $B$ (where $Q$ is the permutation module of right
cosets of $S$ in $G$). 
\end{prop}
\begin{proof}
It suffice to check Eq.~(\ref{eq: cond}) for $h \in G$ and $r \in S$.  Note that $1 - r \in R^+$
and such elements form a $k$-basis.  Also note that $$\cop^{n-1}(1-r) = 1 \otimes \cdots \otimes 1
- r \otimes \cdots \otimes r$$ ($n$ $1$'s and $n$ $r$'s on the right-hand side of the equation).  Eq.~(\ref{eq: cond}) becomes
\begin{eqnarray}
\label{eq: simp}
\sigma^{-1}(h^{-1}r^{-1}, rh) \# h^{-1} r^{-1} \otimes_B (1_D \# rh) && \nonumber \\ &= & \sigma^{-1}(h^{-1},h) \# h^{-1} \otimes_B (1_D \# h) 
\end{eqnarray}
The inverse of the $2$-cocycle Equation~(\ref{cocycle}) is ($\forall x,y,z \in G$)
\begin{equation}
\label{eq: twococycleinverse}
\sigma^{-1}(xy,z) \sigma^{-1}(x,y) = \sigma^{-1}(x,yz) [ x \cdot \sigma^{-1}(y,z)]
\end{equation}
Letting $x = h^{-1}, y = r^{-1}, z = rh$, the left-hand side of Eq.~(\ref{eq: simp}) becomes
$$ \sigma^{-1}(h^{-1},h)[h^{-1} \cdot \sigma^{-1}(r^{-1}, rh)] \sigma(h^{-1}, r^{-1}) \# h^{-1} r^{-1} \otimes_B (1_D \# rh) $$
$$ = (\sigma^{-1}(h^{-1}, h) \# h^{-1}) ( \sigma^{-1}(r^{-1}, rh) \# r^{-1}) \otimes_B (1_D \# rh)$$
$$ = (\sigma^{-1}(h^{-1}, h) \# h^{-1}) \otimes_B (\sigma^{-1}(r^{-1}, rh) \sigma(r^{-1}, rh) \# h)$$
which equals the right-hand side of Eq.~(\ref{eq: simp}).  

The natural module ${}_BA$ is free by a short argument using coset representatives $g_1,\ldots,g_q$
giving a left $R$-basis for $H$.  I.e., $|G: S| = q$, $\sum_{i=1}^q Rg_i = H$ and $\sum_{i=1}^q r_i g_i = 0$
implies each $r_i = 0$.  Then $\sum_{i=1}^q (D \#_{\sigma} R)(1_D \# g_i) = \sum_i D\sigma(R, g_i) \# Rh_i = D \# H$, since $\sigma(x,y)$ is invertible in $D$ for all $x,y \in G$. 
Suppose $$0= \sum_{i=1}^q (d_i \# r_i)(1 \# g_i) = \sum_i d_i \sigma({r_i}\1, g_i) \# {r_i}\2 g_i $$
where each $r_i$ is a linear combination of group elements $s_{ij}$ in $S$.   In this case each
$s_{ij} g_i \in Sg_i$ and it follows from the partition of $G$ into left cosets and the equation just above that each $d_i \# r_i = 0$.

Finally $Q = H / R^+H \cong k[S \setminus G]$ by noting that $R^+$ has $k$-basis $\{ 1 - s \, : \, s \in S \}$.  
\end{proof}

\begin{theorem}
Let $G$ be a group with subgroup $H$ of finite index, $A$ a left twisted $G$-module $k$-algebra, and  $A \#_{\sigma} G$ an (associative) crossed product.  Then h-depth satisfies
\begin{equation}
\label{eq: hdepthineq}
d_h(A \#_{\sigma} H, A \#_{\sigma} G) \leq d_h(kH, kG).
\end{equation}
\begin{proof}
Follows from the previous proposition and  inequality (\ref{eq: mainresult}). 
\end{proof}  
\end{theorem}
\subsection{General results for Hopf algebra $H$} We may extend Eq.~(\ref{eq: hdepthineq})  to a finite-dimensional Hopf algebra extension and crossed product algebra extension as follows.  
\begin{theorem}
Suppose $H$ is a finite-dimensional Hopf algebra, with $R$ a Hopf subalgebra, $A = D \#_{\sigma} H$, $B = D \#_{\sigma} R$ and $Q = H / R^+H$ as above.  Then 
\begin{equation}
d_h(B,A) \leq d_h(R,H).
\end{equation}
\end{theorem}
\begin{proof}
The proof may be made to follow from \cite[3.6]{Sch} and Eq.~(\ref{eq: mainresult}), but we provide some more details as a convenience to the reader.  
Note first  that $\beta: A \otimes_B A \rightarrow A \otimes Q$, $x \otimes y \mapsto xy\0 \otimes \overline{y\1}$ in the proof of Proposition~\ref{prop-sch} is shown to be surjective from the formal
inverse $\beta^{-1}$ defined there and the equation $\beta \circ \beta^{-1} = \id_{A \otimes Q}$.  

That $\beta$ is injective follows from \cite{Sch} in the following way using norms of augmented Frobenius algebras, freeness of $H$ over $R$ and the augmentation of the convolution algebra $Q^*$ induced from the grouplike element $\overline{1_H} \in Q$.  Let $B$ denote $A^{\mbox{co}\, Q}$
for this argument.  Let $\lambda_H, \lambda_R$ be nonzero left integrals on $H$ and $R$, respectively.  Let $\Gamma'  \in R$ satisfy $\lambda_R \Gamma' = \eps$, so that $\Gamma'$ is
a nonzero right integral in $R$.  Define $\lambda: Q \rightarrow k$ by $\lambda(\overline{h}) = \lambda_H(\Gamma' h)$, and note the left integral property, $\overline{h\1} \lambda(\overline{h\2})$ for all $h \in H$
\cite[p.\ 307, (i)]{Sch}.  By \cite[p.\ 307, (ii)]{Sch}, there is $\Lambda \in H$ such that $\eps_Q = 
\Lambda \lambda$,  which follows  from expressing a right norm $\Gamma$ of $\lambda_H$
as $\Gamma = h \Gamma'$, then applying a Nakayama automorphism $\alpha$ of $H$
to express $\Lambda = \alpha(h)$.  Next define as in \cite[p.\ 303, (1)]{Sch} $f: {}_BA \rightarrow {}_BB$ by $f(a) = a\0 \lambda(\overline{a\1})$ using the left integral property. 

In order to continue,  note that $\mbox{can}: A \otimes A \rightarrow A \otimes H$, given
by $x \otimes y \mapsto xy\0 \otimes y\1$ is surjective, since given $a \otimes h \in A \otimes H$,
$\mbox{can}(a(\sigma^{-1}(S(h\2), h\3) \# S(h\1)) \otimes 1_D \# h\4) = a \otimes h$
by the computation on \cite[p.\ 109]{M}.  It follows from the bijectivity of the antipode
that $\mbox{can'}: A \otimes A \rightarrow A \otimes H$ given by $x \otimes y \mapsto x\0 y \otimes x\1$ is also surjective; cf. \cite[p.\ 124]{M}.  Let $\mbox{can'}(\sum_i r_i \otimes \ell_i) = 1_A \otimes \Lambda$.  Then  \cite[p.\ 303, (2)]{Sch} shows that $a = \sum_i f(ar_i)\ell_i$ for each $a \in A$. 
Applying this projectivity equation to $\beta(\sum_k x_k \otimes_B y_k) = 0$, \cite[p.\ 303, (3)]{Sch} shows that
 $\sum_i x_k \otimes_B y_k = 0$. 

Since $H$ is a free $R$-module, a faithfully flat descent along the crossed product extension shows
that $A \#_{\sigma} R = A^{\mbox{co}\, Q}$. 
Since $\beta: A \otimes_B A \stackrel{\cong}{\longrightarrow} A \otimes Q$ as $A$-bimodules,
it follows from Eq.~(\ref{eq: mainresult}), and the proof preceding it, that the inequality in
the theorem holds.    
\end{proof}

\begin{cor}
\label{cor-core}
Suppose $N$ is a normal subgroup of a finite group $G$ contained in a subgroup $H \leq G$.  Then the h-depth satisfies the equality,  $ d_h(kH,kG) = d_h(k[H/N], k[G/N])$. 
\end{cor}
\begin{proof}
This follows from Corollary~\ref{cor-cue}, since in either case $Q \cong k[G \setminus H]$, while $N$ acts trivially
on this $G$-module.  Note that $I = kGkN^+$ is a Hopf ideal in $kG$ (generated by $\{ g - gn \| \, g \in G, n \in N \}$), which annihilates $Q$, and we apply Lemma~\ref{lemma-modI}.  

A second proof is to apply the observation in  Example~\ref{example-groupascrossedproduct} that one has $kG = kN \#_{\sigma} k[G/N]$ and similarly $kH = kN \#_{\sigma} k[H/N]$.  Apply now the inequality~(\ref{eq: mainresult}) to obtain $d_h(kH,kG) \leq d_h(k[H/N], k[G/N])$.   The inequality $d_h(k[H/N], k[G/N]) \leq d_h(kH,kG)$ has a proof very similar to the  proof of Theorem~\ref{prop-sigma} using the relatively nice Hopf ideal $I = kG kN^+$.  
\end{proof}

\subsection{Acknowledgements}  The authors  thank Dr.\ C. Young for stimulating conversations about depth-preserving algebra homomorphisms. Research for this paper was funded by the European Regional Development Fund through the programme {\tiny COMPETE} 
and by the Portuguese Government through the FCT  under the project 
\tiny{ PE-C/MAT/UI0144/2013.nts}.


\begin{thebibliography}{XXXXXX}
\bibitem{AM}N.~Andruskiewitsch and J.M.~Mombelli, On module categories over finite-dimensional Hopf algebras, \textit{J.\ Algebra} \textbf{314} (2007), 383--418.
\bibitem{BW}T.~Brzezinski and R.~Wisbauer, \textit{Corings and Comodules}, Lecture Note Series \textbf{309}, L.M.S.,  Cambridge University Press, 2003.
\bibitem{BDK} R.~Boltje, S.~Danz and B.~K\"ulshammer, On the depth of subgroups and group algebra extensions, \textit{J.\ Algebra} \textbf{335} (2011), 258--281.  
\bibitem{BK} R.~Boltje and B.~K\"ulshammer,
On the depth 2 condition for group algebra and Hopf algebra extensions, 
\textit{J.\ Algebra} \textbf{323} (2010), 1783-1796.
\bibitem{BK2}{R.~Boltje and B.~K\"ulshammer,
Group algebra extensions of depth one, \textit{Algebra Number Theory} \textbf{5}
(2011), 63-73.} 
\bibitem{B2012}S.~Burciu, Kernels of representations and coideal subalgebras of Hopf algebras, \textit{Glasgow Math.\ J.} \textbf{54} (2012), 107--119. 
\bibitem{BuK} S.~Burciu and L.~Kadison, Subgroups of depth three,  \textit{Surv.\ Diff.\ Geom.} \textbf{XV} (2011), 17--36. 
\bibitem{BKK} S.~Burciu, L.~Kadison and B.~K\"ulshammer, On subgroup depth
(with an appendix by B.~K\"ulshammer and S.~Danz),  \textit{I.E.J.A.} \textbf{9} (2011), 133--166. 
\bibitem{CMZ}S.~Caenepeel, G.~Militaru and S.~Zhu, Frobenius and separable functors for generalized module categories and nonlinear equations, Lect.\ Notes Math.\ \textbf{1787}, Springer, 2002. 
\bibitem{CR}C.W.~Curtis and I.~Reiner, \textit{Methods of Representation Theory}, Vol.\ 1, Wiley
Interscience, 1981. 
\bibitem{D}S.~Danz, The depth of some twisted group extensions, \textit{Comm.\ Alg.} \textbf{39} (2011),   1--15.
\bibitem{EO} P.~Etingof and V.~Ostrik, Finite tensor categories, \textit{Moscow J.\ Math.} \textbf{4} (2004),  627--654, 782--783.
\bibitem{Fe}W.~Feit,  \textit{The Representation Theory of Finite Groups}, North-Holland, 1982. 
\bibitem{F}T.~Fritzsche, The depth of subgroups of $\mathrm{PSL}(2,q)$, \textit{J.\ Algebra} \textbf{349} (2011), 217--233. \textit{Ibid} II, \textit{J. Algebra} \textbf{381} (2013), 37--53.
\bibitem{FKR}T.~Fritzsche, B.~K\"ulshammer and C.~Reiche, The depth of Young subgroups of symmetric groups, \textit{J.\ Algebra} \textbf{381} (2013),  96--109. 
\bibitem{HKY}A.~Hernandez, L.~Kadison and C.~Young, Algebraic quotient modules and subgroup depth, \textit{Abh.\ Math.\ Semin.\ Univ.\ Hamburg} \textbf{84} (2014), 267--283.  
\bibitem{HHP}L.~H\'ethelyi, E.~Horv\'ath and F.~Pet\'enyi, The depth of subgroups of Suzuki groups, arXiv preprint \texttt{1404.1523}.
\bibitem{H}{G. Hochschild,
Relative homological algebra, \textit{Trans.\ A.M.S.}  \textbf{82 } (1956), 246--269.}
\bibitem{NEFE} L.~Kadison, 
\textit{New examples of Frobenius extensions}, 
University Lecture Series \textbf{14},
Amer.\ Math.\ Soc., Providence,  1999. 
\bibitem{LK2004}L.~Kadison, Depth two and the Galois coring, in: \textit{Noncommutative geometry and representation theory in mathematical physics}, eds.\ J.\ Fuchs, A.A.\ Stolin \textit{et al},  Contemp.\ Math.\ \textbf{391}, A.M.S.,
Providence, 2005, 149--156. 
\bibitem{LK2011}L.~Kadison, Odd H-depth and H-separable extensions, \textit{Cen.\ Eur.\ J.\ Math.}
\textbf{10} (2012), 958--968.
\bibitem{LK2012}L.~Kadison, Subring depth, Frobenius extensions and towers, \textit{Int.\ J.\ Math.\ \& Math.\ Sci.} \textbf{2012}, article 254791. 
\bibitem{K2013}L.~Kadison, Hopf subalgebras and tensor powers of generalized permutation modules, \textit{J.\ Pure Appl.\ Alg.} \textbf{218} (2014), 367--380.
\bibitem{KS}L.~Kadison and K.~Szlachanyi, Bialgebroid actions on depth two extensions and duality, 
    \textit{Adv.\ in Math.} \textbf{179} (2003), 75-121.
\bibitem{KY}L.~Kadison and C.J.~Young, Subalgebra depths within the path algebra of an acyclic quiver, 
in: \textit{Algebra, Geometry and Mathematical Physics} (AGMP, Mulhouse,
France, Oct. 2011), Springer Proc.\
Math.\ Stat.\ \textbf{85}, eds. Makhlouf, Stolin et al,  2014, 83--97.
\bibitem{Ma}S.~Majid, \textit{Foundations of Quantum Group Theory}, Cambridge University Press, 1995.
\bibitem{M}S.~Montgomery,
\textit{Hopf Algebras and their Actions on Rings}, \textit{C.B.M.S.}
\textbf{82}, A.M.S., 1993.  
\bibitem{P}D.S.~Passman, \textit{The Algebraic Structure of Group Rings}, Dover, 2011, originally publ.\ 1977, revised 1985. 
\bibitem{Sch} H.-J.~Schneider, Normal basis and transitivity of crossed products for Hopf algebras, \textit{J.\ Algebra} \textbf{152} (1992), 289--312. 
\bibitem{S}S.~Skryabin, Projectivity and freeness over comodule algebras, \textit{Trans.\ A.M.S.}
\textbf{359} (2007),  2597--2623.
\bibitem{MS}M.~Szamotulski, Galois Theory for H-extensions,  Ph.D. Thesis, Technical U.\ Lisbon, 2013.
\bibitem{PQ}D.S.~Passman and D.~Quinn, Burnside's theorem for Hopf algebras, \textit{Proc.\ A.M.S.} \textbf{123} (1995), 327--333.
 \bibitem{U}K.-H. Ulbrich, On modules induced or coinduce from Hopf subalgebras, \textit{Math.\ Scand.} \textbf{67} (1990), 177--182.
\bibitem{Y}C.J.~Young, Depth theory of Hopf algebras and smash products, Uni.\ Porto Ph.D.\ Dissertation, 2014.  
\end{thebibliography}
\end{document}